\documentclass[11pt]{amsart}
\usepackage[margin=1in]{geometry}
\usepackage{amsmath}
\usepackage{amsfonts,amsmath}
\usepackage{paralist}
\usepackage[colorlinks=true]{hyperref}
\hypersetup{urlcolor=blue, citecolor=red}
 \numberwithin{equation}{section}

\newcommand{\io}{\int_{\Omega}}
\newcommand{\ib}{\int_{B_r(y)}}

\newcommand{\mnp}{(\m\cdot\nabla p) }

\newcommand{\dy}{\textup{dist}(y,\partial\Omega)}

  \newcommand{\m}{\mathbf{m}}

\newcommand{\ot}{\Omega_T }
\newtheorem{theorem}{Theorem}[section]

\newtheorem*{main}{Main Theorem}
\newtheorem{lemma}[theorem]{Lemma}

\theoremstyle{definition}
\newtheorem{definition}[theorem]{Definition}

\newcommand{\ep}{\varepsilon}

\title[a biological network formulation model
] 
      {Global existence of strong solutions to a biological network formulation model in $2+1$ dimensions}

\author[Xiangsheng Xu]{}

 \keywords{ biological network formulation, cubic nonlinearity, De Giorgi iteration method.
 {\it Discrete \& Continuous Dynamical Systems - A}, to appear.
}
 \email{xxu@math.msstate.edu}

\begin{document}
\maketitle

\centerline{\scshape Xiangsheng Xu}
\medskip
{\footnotesize
 \centerline{Department of Mathematics \& Statistics}
   \centerline{Mississippi State University}
   \centerline{ Mississippi State, MS 39762, USA}
} 

\bigskip

\begin{abstract}
In this paper we study the initial boundary value problem for the system\\ $-\mbox{{div}}\left[(I+\mathbf{m} \mathbf{m}^T)\nabla p\right]=s(x),\ \  \mathbf{m}_t-\alpha^2\Delta\mathbf{m}+|\mathbf{m}|^{2(\gamma-1)}\mathbf{m}=\beta^2(\mathbf{m}\cdot\nabla p)\nabla p$ in two space dimensions. This problem has been proposed as a continuum model for biological transportation networks. The mathematical challenge is due to the presence of cubic nonlinearities, also known as trilinear forms, in the system. We obtain a weak solution $(\mathbf{m},p) $ with both $|\nabla p|$ and $|\nabla\mathbf{m}|$ being bounded. The result immediately triggers a bootstrap argument which can yield higher regularity for the weak solution. This is achieved by deriving an equation for $v\equiv(I+\mathbf{m} \mathbf{m}^T)\nabla p\cdot\nabla p$, and then suitably applying the De Giorge iteration method to the equation.
\end{abstract}
\section{Introduction}
Continuum models for biological transportation networks have received tremendous attention recently. We refer the reader to \cite{ABHMS} for a rather comprehensive survey of the subject. The most well known model is the one proposed by Hu and Cai \cite{H,HC}. It describes the pressure field of a network using a Darcy's type equation and the dynamics of the conductance network under pressure force effects. More precisely, 
let $\Omega$ be the network region, a bounded domain in $\mathbb{R}^N$ with  boundary $\partial\Omega$, and $T$ a positive number. Set $\ot=\Omega\times(0,T)$. Then the scalar pressure function $p=p(x,t)$ and the $N$-dimensional conductance vector field $\m=\m(x,t)$ satisfy the system
\begin{align}
-\mbox{{div}}\left[(I+\m \m^T)\nabla p\right]&=s(x)\ \ \ \mbox{in $\ot$,}\label{e1}\\
\m_t-\alpha^2\Delta \m+|\m|^{2(\gamma-1)}\m&=\beta^2\mnp\nabla p\ \ \ \mbox{in $\ot$}\label{e2}
\end{align}
coupled with the initial boundary conditions
\begin{align}
p=|\m|&=0 \ \ \ \mbox{ on $\Sigma_T\equiv\partial\Omega\times(0,T)$} ,\label{e3}\\
\m(x,0)&=\m_0(x)\ \ \ \mbox{on $\Omega$}. \label{e4}
\end{align} 
Here the function $s(x)$ is the time-independent source term. Values of the parameters $\alpha, \beta$, and $\gamma$ are
determined by the particular physical applications one has in mind. For example, in blood vessel systems, we have $\gamma=\frac{1}{2}$, while in leaf venation   $\gamma>\frac{1}{2}$ is likely \cite{HMPS}. 

Many aspects of the model have been investigated. A result in \cite{HMP} asserts that  \eqref{e1}
-\eqref{e4} has a weak solution, provided that, in addition to assuming $s(x)\in L^2(\Omega)$ and $\alpha, \beta>0, \gamma \geq 1$, we also have
\begin{equation*}
\m_0\in\left( W^{1,2}_0(\Omega)\cap L^{2\gamma}(\Omega)\right)^N.
\end{equation*} 
The notion of a weak solution in \cite{HMP} is defined as follows:
\begin{definition}
	A pair $(\m, p)$ is said to be a weak solution if:
	\begin{enumerate}
		\item[(D1)] $\m\in L^\infty\left(0,T; \left(W^{1,2}_0(\Omega)\cap L^{2\gamma}(\Omega)\right)^N\right),\ \partial_t\m\in L^2\left(0,T; \left(L^2(\Omega)\right)^N\right),\  p\in L^\infty(0,T; W^{1,2}_0(\Omega)),\\  \mnp \in L^\infty(0,T;  L^{2}(\Omega))$;
		\item[(D2)] $\m(x,0)=\m_0$ in $C\left([0,T]; \left(L^2(\Omega)\right)^N\right)$;
		\item[(D3)] Equations \eqref{e1} and \eqref{e2} are satisfied in the sense of distributions.
	\end{enumerate}
\end{definition}
The proof in \cite{HMP} was based upon the following two equations:
\begin{eqnarray}
\lefteqn{\frac{1}{2}\int_{\Omega}|\m(x,\tau)|^2dx+D^2\int_{\Omega_\tau}|\nabla \m|^2dxdt+\beta^2\int_{\Omega_\tau} \mnp^2dxdt}\nonumber\\
&&+\int_{\Omega_\tau}|\m|^{2\gamma}dxdt+2\beta^2	\int_{\Omega_\tau} |\nabla p|^2dxd\tau\nonumber\\
&=&\frac{1}{2}\int_{\Omega}|\m_0|^2dx+2\beta^2\int_{\Omega_\tau} s(x)pdxdt,\label{me1}\\
\lefteqn{\int_{\Omega_\tau}|\partial_t\m|^2dxdt+\frac{\alpha^2}{2}\int_{\Omega}|\nabla \m(x,\tau)|^2dx+\frac{\beta^2}{2}\int_{\Omega}\mnp^2 dx}\nonumber\\
&&+\frac{\beta^2}{2}\int_{\Omega}|\nabla p|^2dx+\frac{1}{2\gamma}\int_{\Omega}|\m|^{2\gamma}dx\nonumber\\
&=&\frac{\alpha^2}{2}\int_{\Omega}|\nabla \m_0|^2dx+\frac{\beta^2}{2}\int_{\Omega}(\m_0\cdot \nabla p_0)^2dx+\frac{1}{2\gamma}\int_{\Omega}|\m_0|^{2\gamma}dx\nonumber\\
&&+\frac{\beta^2}{2}\int_{\Omega}|\nabla p_0|^2dx,\label{f5}
\end{eqnarray}
where $\tau\in (0, T], \Omega_\tau=\Omega\times(0,\tau)$, and $p_0$ is the solution of the boundary value problem
\begin{eqnarray}
-\mbox{div}[(I+\m_0 \m_0^T)\nabla p_0] &=& s(x),
\ \ \ \mbox{in $\Omega$,}\label{po1}\\
p_0&=& 0\ \ \ \mbox{on $\partial\Omega$.}\label{po2}
\end{eqnarray}
Finite time extinction or break-down of solutions in the spatially one-dimensional setting for certain ranges of the relaxation exponent $\gamma$ was carefully studied in \cite{HMPS}. Further modeling analysis and numerical results can be found in \cite{AAFM}. We also mention that the question of existence in the case where $\gamma=\frac{1}{2}$ is addressed in \cite{HMPS}. The novelty here is that the term $|\m|^{2(\gamma-1)}\m$ is not continuous at $\m=0$. 
Nonetheless, the general regularity theory remains fundamentally incomplete. In particular, it is not known whether or not weak solutions develop singularities when the space dimension $N$ is bigger than or equal to $2$.
If the space dimension $N$ is three, the initial value problem for the system \eqref{e1}-\eqref{e2} has been studied in \cite{L}, where the local existence of a strong solution and global existence of such a solution for small data are established. In addition, the author obtained a condition  which a strong solution must satisfy if it blew up in finite time. However, the author specifically mentioned that his method there was not applicable to the case where $N=1$ or $2$. For $N=2$ the same initial value problem was considered in \cite{SL}. Here the authors obtained a similar blow-up criterion to that in \cite{L} and the global existence of a strong solution under the additional assumptions that $\alpha$ is sufficiently large and $\gamma\geq 1$. As for the initial-boundary value problem for \eqref{e1} and \eqref{e2},  Jian-Guo Liu and the author \cite{LX} obtained a partial regularity theorem for \eqref{e1}-\eqref{e4}. It states that the parabolic Hausdorff dimension of the set of singular points can not exceed $N$, provided that $N\leq 3$. A different form of partial regularity is obtained in \cite{X6}. If $N=2$, then it is shown in \cite{X5} that $p$ is continuous in the space variables and $(\m, p)$ becomes a classical solution under the additional assumption that it is time-independent.

We study the regularity properties of weak solutions to problem \eqref{e1}-\eqref{e4}
for $N=2$ under the assumptions that the given functions $s(x), \m_0(x)$ and physical parameters $\alpha, \beta, \gamma$ have properties:
\begin{enumerate}
	\item[(H1)] $s(x)\in L^{\infty}(\Omega)\cap W^{1,q}(\Omega)$ for some $q>1$; 
	\item[(H2)] $\alpha, \beta\in (0, \infty), \gamma\in (\frac{1}{2}, \infty)$; and
	\item[(H3)] $\m_0(x)\in \left(W^{1,\infty}_0(\Omega)\right)^2$.
\end{enumerate}

We are ready to state our main result. 
\begin{main} Let $\Omega$ be a bounded domain in $\mathbb{R}^2$ with $C^{3+\varepsilon}$ boundary $\partial\Omega$ for some $\varepsilon\in (0,1)$. Assume that (H1)-(H3) hold. Then for each $T>0$ there exists a weak solution $(\m, p)$ to \eqref{e1}-\eqref{e3}
	with
	\begin{equation}\label{mcon}
	|\nabla p|, |\nabla\m|\in L^\infty(\ot).
	\end{equation} 
\end{main}

Since $T$ can be any positive number, the above theorem asserts that  blow-up in $|\m|$ \cite{SL} does not occur in finite time, the result in \cite{X5} can be extended to the time-dependent case, and the singular set in \cite{LX} is empty when the space dimension $N$ is $2$. 
It also implies  that in this case we can extend the local solution given by Theorem 1.7 in  \cite{X6} in the time direction as far away as we want. It is not difficult to see from the proof of \eqref{nmb2} below that \eqref{mcon} is a consequence of (D4) in \cite{X6}. In particular, we do not impose any additional restrictions on the size of the given data except those already enumerated in (H1)-(H3). Thus our theorem is different from  the results in \cite{L,SL}. 

Of course, our main theorem can start a bootstrap argument which results in even higher regularity. In fact, our main theorem implies the existence of the so-called strong solution \cite{L,X6} in the sense that the system \eqref{e1}-\eqref{e2} is satisfied a.e. on $\ot$. Indeed,  
it immediately follows from the main theorem that $\beta^2\mnp\nabla p-|\m|^{2(\gamma-1)}\m\in L^\infty
(\ot)$. The classical regularity theory for the heat equation (\cite{LSU}, Chapter IV) asserts that 
\begin{equation}\label{r3}
|\m_t|,\ \  |\Delta \m|^2\in L^s(\ot)\ \ \mbox{for each $s\geq 1$.} 
\end{equation} Thus we can write \eqref{e1} in the form
\begin{equation}\label{e11}
\mbox{tr}\left[(I+\m \m^T)\nabla^2 p\right]+\mbox{div}	(I+\m \m^T)\nabla p=-s(x). 
\end{equation}
This enables us to conclude from the classical regularity theory for elliptic equations that  $p\in L^\infty(0,T;W^{2,s})$ for each $s\geq 1$. 
Moreover, we can differentiate this equation with respect $x_i$ to derive an equation for $p_{x_i}$. This combined with our high regularity assumption on the boundary $\partial\Omega$ ensures that we even have
\begin{equation}\label{r2}
p\in  W^{3,s}(\Omega)\ \ \mbox{for some $s>1$ and a.e. $t\in (0, T)$.} 
\end{equation} 
This fact will be needed later on.  Note that we only assume that $\gamma>\frac{1}{2}$. As a result, we cannot
differentiate the term $|\m|^{2(\gamma-1)}\m$. This prevents us from obtaining any estimates for the third order partial derivatives of $\m$. 
Note that the elliptic coefficients in \eqref{e1} satisfy
$$|\xi|^2\leq \left((I+\m \m^T)\xi\cdot\xi\right)=|\xi|^2+(\m\cdot\xi)^2\leq (1+|\m|^2)|\xi|^2\ \ \ \mbox{for all $\xi\in \mathbb{R}^2$.}$$
That is, \eqref{e1} is only singular. This enables us to show that $p$ is bounded \cite{LX}. In fact, we have $p\in L^\infty(0,T; C_{\mbox{loc}}(\Omega))$ \cite{X5}. Unfortunately, this is not enough to trigger a bootstrap argument. We must have the H\"{o}lder continuity of $p$ in the space variables to obtain the boundedness of $\m$ (see Lemma \ref{phc1} below). Instead of trying to bridge this gap, we directly go after the boundedness of $\nabla p$. This is motivated by a result in \cite{A} where the author considered an elliptic equation of the
form
\begin{equation}\label{r1}
a_{ij}u_{x_ix_j}+b_iu_{x_i}=0,\ \ i,j=1,2.
\end{equation} 
Here we have employed the Einstein summation convention. That is, repeated indices are implicitly summed over. Denote by $A$ the elliptic coefficient matrix in \eqref{r1}, i.e., $(A)_{ij}=a_{ij}$.
An equation for $\ln \left(A\nabla u\cdot\nabla u\right)$ was derived in \cite{A} to study critical points of $u$. In our case
\begin{equation}\label{adf}
A=I+\m \m^T.
\end{equation}
The key observation is that we can extend the argument in \cite{A} and derive an equation for $v\equiv\left(A\nabla p\cdot\nabla p\right)$. To be specific, let us introduce the following quantities
\begin{eqnarray}
\mathbf{G}&=&v^{-1}\left(\begin{array}{c}
A_{x_1}\nabla p\cdot\nabla p\\
A_{x_2}\nabla p\cdot\nabla p
\end{array}\right),\label{gdef}\\
A_1&=&\left(\begin{array}{cc}
a_{11}\nu_1&a_{12}a_{11}p_{x_1}-(a_{11}a_{22}-2a_{12}^2)p_{x_2}\\
a_{11}\nu_2&a_{22}\nu_1
\end{array}\right),\label{a1}\\
A_2&=&\left(\begin{array}{cc}
a_{11}\nu_2&a_{22}\nu_1\\
-(a_{22}a_{11}-2a_{12}^2)p_{x_1}+a_{12}a_{22}p_{x_2}&a_{22}\nu_2
\end{array}\right),\label{a2}\\
A_3
&=&-\left(\begin{array}{cc}
\nu_1^2&\nu_1\nu_2\\
\nu_1\nu_2&\nu_2^2
\end{array}\right),\label{a3}\\
\mathbf{H}&=&\frac{1}{\mbox{det}(A)v}\left(\begin{array}{c}(\nabla p)^TA_1\\(\nabla p)^TA_2\end{array}\right)\nabla\mbox{det}(A)-A\mathbf{G},\label{vhdef}\\
\mathbf{K}&=&A\mathbf{G}+2v^{-1}wA\nabla p,\label{kdef}\\
h&=&\left(2v^{-1}wA\nabla p-\frac{1}{\mbox{det}(A)v}\left(\begin{array}{c}(\nabla p)^TA_1\\(\nabla p)^TA_2\end{array}\right)\nabla\mbox{det}(A)+A\mathbf{G}\right)\cdot \mathbf{G}\nonumber\\
&&+\frac{2w}{\mbox{det}(A)v^2}\left(A_3\nabla p\cdot\nabla\mbox{det}(A)\right),\label{hdef}
\end{eqnarray}
where $\nu_1=a_{11}p_{x_1}+a_{12}p_{x_2}, \nu_2=a_{12}p_{x_1}+a_{22}p_{x_2}$, i.e.,
\begin{equation}\label{vddef}
\left(\begin{array}{c}
\nu_1\\
\nu_2
\end{array}\right)=A\nabla p
\end{equation} 
because our $A$ is symmetric.
Then $ v $ satisfies the equation
\begin{equation}\label{wine11}
\textup{div}\left(\frac{1}{ v }A\nabla v \right)= \frac{1}{ v }\mathbf{H}\cdot\nabla v +h+\textup{div}\mathbf{K}\ \ \mbox{in $\{|\nabla p|>0\}$}.
\end{equation}
The downside is that this equation  is both degenerate and singular. We overcome these singularities by suitably modifying the classical De Giorge iteration method. Here we explore the fact that $\mathbf{H}, h, \mathbf{K}$ are only bounded by $|\m|, |\nabla\m|$ for large $v$.
Even though the derivation of \eqref{wine11} is inspired by a result in \cite{A}, there are some major differences. The most prominent one is that we have not been able to impose the normalization condition $a_{11} a_{22}-a_{12}^2=1$
as  in \cite{A}. 
Doing so would have changed the smallest eigenvalue of the coefficient matrix to $\frac{1}{\sqrt{1+|\m|^2}}$, which is not bounded away from $0$ below because we do not have the a priori knowledge that $\m$ is bounded. The resulting estimate for $ v $ would be useless to us. As we shall see, not being able to normalize the coefficient matrix causes many complications. 

The idea of deriving an equation for the modulus of the gradient of solutions of an elliptic equation can be traced back to \cite{B}. It has proved to be a powerful tool in obtaining upper bounds for the gradient \cite{PS,S}.

The rest of the paper is organized as follows. Section 2 is largely preparatory. Here we collect some relevant known results for later use. 
In Section 3 we first derive \eqref{wine11}. To justify all the calculations in the derivation, we need to assume that $(\m,p)$ is sufficiently regular. Here we use the local existence result in \cite{X6}. That is, there is a weak solution  $(\m,p)$ to \eqref{e1}-\eqref{e4} satisfying the conclusion of our main theorem at least for $T$ sufficiently small. As we discussed earlier, this solution actually possesses much higher regularity, which will be enough for our purpose. 
The actual proof of the main theorem is achieved in two stages. First we show that $\sup_{0\leq t\leq T}\|\nabla p\|_{\infty,\Omega}$ is bounded by  a power of $\sup_{0\leq t\leq T}\|\nabla\m\|_{\infty,\Omega}$. 
This is done via \eqref{wine11} and the De Giorge iteration method. 
Then we prove that $\|\nabla\m\|_{\infty,\ot}$ is also bounded $\|\nabla p\|_{\infty,\ot}^2$. Our theorem can be established by choosing $T$ suitably small. 
Then we show that we can extend our solution in the time direction as far away as we want.

Note that our assumption on $s(x) $ is more than necessary. If we check our calculations carefully, it is enough to assume $s(x)\in L^q(\Omega)$ for $q$ sufficiently large. Then we must use functions in $L^{\infty}(\Omega)\cap W^{1,q}(\Omega)$ to approximate $s(x)$ in order to gain enough regularity for $p$ to justify the calculations in the derivation of \eqref{wine11}.

Finally, let us make some remarks about notations. The letter $c$ is used to denote a positive number whose value can be computed from given data. The capital letters such as $A, B, \cdots$ are often used to represent $2\times 2$ matrices. The $ij$-entry of $A$ is denoted by $a_{ij}$. The boldface letters are used to denote vector quantities. The $i$-th entry of $\mathbf{F}$ is $f_i$.

\section{Preliminary results} In this section we first collect some formulas about differentiating matrix-valued functions. Then we prove that local H\"{older} continuity of $p$ in the space variables implies the local boundedness of $\m$.

Denote by $M^{2\times 2}$ the space of all $2\times 2$ matrices. 
We invoke the following notation conventions
\begin{eqnarray*}
	A:B&=& a_{ij}b_{ij}\ \ \ \mbox{for $A, B\in M^{2\times 2}$,}\\
	\mathbf{G}\otimes\mathbf{F}&=&\mathbf{G}\mathbf{F}^T,\ \ \left(\mathbf{G}\cdot\mathbf{F}\right)=\mathbf{G}\cdot\mathbf{F}=\mathbf{G}^T\mathbf{F} \ \ \mbox{for two (column) vectors $\mathbf{G}, \mathbf{F}$}.
\end{eqnarray*}
If $A(x)$ is a matrix-valued function then
\begin{eqnarray*}
	\mbox{div}A(x)&=&\mbox{the row vector whose $i$-th entry is the divergence of the  $i$-th column of $A$ } \nonumber\\
	&=&(\mbox{div}\mathbf{A}_1,\mbox{div}\mathbf{A}_2).
\end{eqnarray*}
When $\mathbf{G(x)}$ is a vector-valued function, then
\begin{eqnarray*}
	\nabla\mathbf{G(x)}&=&\mbox{the $2\times 2$ matrix whose $ij$-entry is $(g_j(x))_{x_i}$}\nonumber\\
	& =&(\nabla g_1, \nabla g_2).
\end{eqnarray*}
Denote by $\nabla^2p$  the Hessian of $p$. Then we have
\begin{equation*}
\nabla|\nabla p|^2= 2\nabla^2p\nabla p.
\end{equation*}
The following identities will be frequently used
\begin{eqnarray}
\nabla \left(\mathbf{F}\cdot\mathbf{G}\right)&=&\nabla \mathbf{F}\mathbf{G}+\nabla\mathbf{G}\mathbf{F},\label{form1}\\
\mbox{div}\left(A\mathbf{F}\right)&=& A:\nabla\mathbf{F} +\mbox{div}A\mathbf{F},\label{form2}\\
\nabla\left(A\mathbf{F}\right)&=& \nabla\mathbf{F}A^T+\left(
A_{x_1}\mathbf{F},
A_{x_2}\mathbf{F}
\right)^T,\label{form3}\\
\mbox{div}(pA)&=&p\mbox{div}A+(\nabla p)^TA.\label{form4}
\end{eqnarray}
We also need the interpolation inequality
\begin{equation}\label{inter}
\|u\|_q\leq \varepsilon\|u\|_r+\varepsilon^{-\mu}\|u\|_\ell,
\end{equation}
where $1\leq \ell\leq q\leq r$ with $\mu=\left(\frac{1}{\ell}-\frac{1}{q}\right)/\left(\frac{1}{q}-\frac{1}{r}\right)$.

The next lemma deals with sequences of nonnegative numbers
which satisfy certain recursive inequalities.
\begin{lemma}\label{ynb}
	Let $\{y_n\}, n=0,1,2,\cdots$, be a sequence of positive numbers satisfying the recursive inequalities
	\begin{equation*}
	y_{n+1}\leq cb^ny_n^{1+\alpha}\ \ \mbox{for some $b>1, c, \alpha\in (0,\infty)$.}
	\end{equation*}
	If
	\begin{equation*}
	y_0\leq c^{-\frac{1}{\alpha}}b^{-\frac{1}{\alpha^2}},
	\end{equation*}
	then $\lim_{n\rightarrow\infty}y_n=0$.
\end{lemma}
This lemma can be found in (\cite{D}, p.12).


\begin{lemma}\label{phc1}
	Assume that the space dimension $N=2$.	If $p\in L^\infty(0,T;C_{\textup{loc}}^{0,\sigma}(\Omega))$ for some $\sigma\in (0,1)$, then $\m$ is locally bounded.
\end{lemma}
\begin{proof} We infer from (D1) that 
	\begin{equation}\label{mint}
	\sup_{0\leq t\leq T}\io|\m|^sdx<\infty\ \ \mbox{for each $s>1$.}
	\end{equation}
	In fact, by Theorem 7.15 in (\cite{GT}, p. 162), there is a positive number $c_0$ such that
	\begin{equation*}
	\io e^{c_0 |\m|}dx< \infty.
	\end{equation*}
	Fix $y\in\Omega$. For $r\in (0,\dy)$ we choose a smooth cutoff function $\xi$ with the properties
	\begin{eqnarray*}
		\xi(x)&=& 1\ \ \mbox{on $B_{\frac{r}{2}}(y)$},\\  
		\xi(x)&=& 0\ \ \mbox{outside $B_{r}(y)$, }\\ 
		0&\leq &\xi\leq 1\ \ \mbox{on $B_{r}(y)$, and }\\
		|\nabla\xi|&\leq &\frac{c}{r}\ \ \mbox{on $B_{r}(y)$. }
	\end{eqnarray*}
	We use $(p-p_{y,r}(t))\xi^2 $ as a test function in \eqref{e1} to get
	\begin{eqnarray*}
		\lefteqn{\ib|\nabla p|^2\xi^2dx+\ib\nabla p (p-p_{y,r}(t))2\xi\nabla\xi dx}\nonumber\\
		&&+\ib\mnp^2\xi^2dx+\ib\mnp\m(p-p_{y,r}(t))2\xi\nabla\xi dx\nonumber\\
		&=&\ib s(x)(p-p_{y,r}(t))\xi^2dx,
	\end{eqnarray*}
	from whence follows
	\begin{eqnarray*}
		\lefteqn{\ib|\nabla p|^2\xi^2dx+\ib\mnp^2\xi^2dx}\nonumber\\
		&\leq &\frac{c}{r^2}\ib(p-p_{y,r}(t))^2dx+\frac{c}{r^2}\ib|\m|^2(p-p_{y,r}(t))^2dx\nonumber\\
		&&+\ib s(x)(p-p_{y,r}(t))\xi^2dx\nonumber\\
		&\leq &cr^{2\sigma}+cr^{2\sigma-2}\ib|\m|^2dx+cr^{\sigma}\ib s(x)dx\nonumber\\
		&\leq &cr^{2\sigma}+cr^{2\sigma-2+\frac{2}{s}}\left(\ib|\m|^{\frac{2s}{s-1}}dx\right)^{\frac{s-1}{s}}+cr^{\sigma}cr^{\sigma}.
	\end{eqnarray*}
	By choosing $s$ sufficiently close to $1$, we can find a positive number $\varepsilon$ such that
	\begin{equation}\label{phc}
	\ib\mnp^2dx\leq cr^\varepsilon.
	\end{equation}
	Take the dot product of \eqref{e2} with $\m$
	to obtain
	\begin{equation*}
	u_t-\alpha^2\Delta u+2\alpha^2|\nabla\m|^2+2u^\gamma=2\beta^2(\m\cdot\nabla p)^2,
	\end{equation*}
	where
	\begin{equation*}
	u=|\m|^2.
	\end{equation*}
	Consider the problem
	\begin{eqnarray}
	w_t-\alpha^2\Delta w&=&=2\beta^2(\m\cdot\nabla p)^2\ \ \mbox{in $\ot$},\label{we1}\\
	w&=&u\ \ \mbox{on $\partial_p\ot$}.
	\end{eqnarray}
	By the comparison principle, we have
	\begin{equation*}
	u\leq w.
	\end{equation*} 
	The right-hand side term in \eqref{we1} satisfies \eqref{phc}, a result in \cite{Y} asserts that $w$ is H\"{o}lder continuous. This implies the desired result. The proof is complete.
\end{proof}

\section{Boundedness for $\nabla p$ and $\nabla\m$}

In this section we will offer the proof of the main theorem. We shall begin by deriving \eqref{wine11}. As we discussed in the Introduction, we may assume that $(\m,p)$ is a weak solution that already satisfies the conclusion of the main theorem. This means that we begin with the local solution constructed in \cite{X6}. Then it turns out that  the local solution can be extended in the time direction as far away as we want. Therefore, we have
\eqref{r3} and \eqref{r2}, which are enough for our subsequent calculations.

Let $A$ be given as in \eqref{adf}. Recall from \eqref{form2} that
\begin{equation*}
\mbox{div}\left(A\nabla p\right)=A:\nabla^2p+\mbox{div}A\nabla p.
\end{equation*}
We can write \eqref{e1} in the form
\begin{equation}
\mbox{tr}(A\nabla^2p)=A:\nabla^2p=w,\label{eqfp}
\end{equation}
where
\begin{equation}\label{awdf}
w=-\left(\mbox{div}A\nabla p+s(x)\right).
\end{equation}
As in \cite{A}, we introduce the  function
\begin{eqnarray*}
	\phi&=&\ln v.
\end{eqnarray*}

\begin{theorem}\label{efp} Let $\mathbf{G}, \mathbf{H}, h, \mathbf{K}$ be given as before (see \eqref{gdef}, \eqref{vhdef}, \eqref{hdef}, and \eqref{kdef} for their respective definitions).
	Then the function $\phi$ satisfies the equation
	\begin{equation}\label{wine1}
	\textup{div}(A\nabla\phi)= \mathbf{H}\cdot\nabla\phi+h+\textup{div}\mathbf{K}\ \ \mbox{in 	 $\{v>0\}$}.
	\end{equation}
\end{theorem}
\begin{proof}
	The  identity 
	\begin{equation}\label{iden1}
	\mbox{div}(A\nabla\phi)=v^{-1}\mbox{div}(A\mathbf{E})-v^{-2}A\mathbf{E}\cdot\mathbf{E}-v^{-1}A\mathbf{G}\cdot\mathbf{E}+\mbox{div}(A\mathbf{G}),
	\end{equation}
	where
	\begin{equation}\label{edf}
	\mathbf{E}=2\nabla^2pA\nabla p,
	\end{equation}
	in \cite{A} is still valid here.
	To see this, we compute from \eqref{form1} and \eqref{form3} that
	\begin{eqnarray}
	\nabla\phi&=&\frac{1}{v}\nabla v=\frac{1}{v}\nabla \left(A\nabla p\cdot\nabla p\right)\nonumber\\
	&=&\frac{1}{v}\left(\nabla(A\nabla p)\nabla p+\nabla^2pA\nabla p\right)\nonumber\\
	&=&\frac{1}{v}\left(\left(A_{x_1}\nabla p,A_{x_2}\nabla p\right)^T\nabla p+2\nabla^2pA\nabla p\right)\nonumber\\
	&=&\mathbf{G}+v^{-1}\mathbf{E}.\label{etp}
	\end{eqnarray}
	Consequently,
	\begin{eqnarray}
	\mbox{div}(A\nabla\phi)&=&v^{-1}\mbox{div}(A\mathbf{E})-v^{-2}\nabla v\cdot A\mathbf{E}+\mbox{div}(A\mathbf{G})\nonumber\\
	&=&v^{-1}\mbox{div}(A\mathbf{E})-v^{-2}\mathbf{E}\cdot A\mathbf{E}-v^{-1}\mathbf{G}\cdot A\mathbf{E}+\mbox{div}(A\mathbf{G})\nonumber\\
	&=&v^{-1}\mbox{div}(A\mathbf{E})-v^{-2}A\mathbf{E}\cdot\mathbf{E}-v^{-1}A\mathbf{G}\cdot\mathbf{E}+\mbox{div}(A\mathbf{G}).\label{hope3}
	\end{eqnarray}
	The last step is due to the fact that $A$ is symmetric. The first two terms on the right-hand side of the above equation are troubling. One contains third order partial derivatives of $p$, while the other is quadratic in $\mathbf{E}$. It turns out that both terms can be represented in terms of $\mbox{det}(\nabla^2p)$. After we substitute them back into \eqref{hope3}, the $\mbox{det}(\nabla^2p)$ terms get canceled out. We shall do this by finding a suitable formula for the matrix $D$ defined by
	\begin{equation*}
	D=A\nabla^2pA.
	\end{equation*}
	An elementary calculation shows that the four entries of $D$ are as follows
	\begin{eqnarray*}
		d_{11}&=&a_{11}^2p_{x_1x_1}+2a_{11}a_{12}p_{x_1x_2}+a_{12}^2p_{x_2x_2},\\
		d_{12}&=&a_{11}a_{12}p_{x_1x_1}+(a_{11}a_{22}+a_{12}^2)p_{x_1x_2}+a_{22}a_{12}p_{x_2x_2},\\
		d_{21}&=&d_{12},\\
		d_{22}&=&a_{12}^2p_{x_1x_1}+2a_{12}a_{22}p_{x_1x_2}+a_{22}^2p_{x_2x_2}.
	\end{eqnarray*}
	Using \eqref{eqfp},
	we obtain
	\begin{eqnarray*}
		\lefteqn{A\nabla^2pA}\nonumber\\
		&=&\left(\begin{array}{cc}
			a_{11}w-\mbox{det}(A)p_{x_2x_2}&a_{12}w+\mbox{det}(A)p_{x_1x_2}\\
			a_{12}w+\mbox{det}(A)p_{x_1x_2}&
			a_{22}w-\mbox{det}(A)p_{x_1x_1}
		\end{array}\right)\nonumber\\
		&=&wA+\mbox{det}(A)\left(\begin{array}{cc}
			-p_{x_2x_2}&p_{x_1x_2}\\
			p_{x_1x_2}&
			-p_{x_1x_1}
		\end{array}\right)=wA-\mbox{det}(A)\mbox{det}(\nabla^2 p)(\nabla^2 p)^{-1}.
	\end{eqnarray*}
	Now we are in a position to calculate that
	\begin{eqnarray}
	\mbox{div}(A\mathbf{E})&=&2\mbox{div}(A\nabla^2pA\nabla p)\nonumber\\
	&=&2\nabla^2p:(A\nabla^2 pA)+2\mbox{div}(A\nabla^2pA)\nabla p.\label{aef}
	\end{eqnarray}
	Applying the formula for $A\nabla^2pA$ yields
	\begin{eqnarray*}
		\nabla^2p:(A\nabla^2 pA)&=&\nabla^2p:wA-\mbox{det}(A)\mbox{det}(\nabla^2 p)\nabla^2p:(\nabla^2 p)^{-1}\nonumber\\
		&=&w^2-2\mbox{det}(A)\mbox{det}(\nabla^2 p).
	\end{eqnarray*}
	Similarly,
	\begin{eqnarray*}
		\mbox{div}(A\nabla^2pA)\nabla p&=&\mbox{div}(wA)\nabla p+\mbox{div}\left[\mbox{det}(A)\left(\begin{array}{cc}
			-p_{x_2x_2}&p_{x_1x_2}\\
			p_{x_1x_2}&
			-p_{x_1x_1}
		\end{array}\right)\right]\nabla p\nonumber\\
		&=&\mbox{div}(wA\nabla p)-w^2+(\nabla\mbox{det}(A))^T\left(\begin{array}{cc}
			-p_{x_2x_2}&p_{x_1x_2}\\
			p_{x_1x_2}&
			-p_{x_1x_1}
		\end{array}\right)\nabla p.
	\end{eqnarray*}
	Here we have used \eqref{form4} and the fact that $\mbox{div}\left(\begin{array}{cc}
	-p_{x_2x_2}&p_{x_1x_2}\\
	p_{x_1x_2}&
	-p_{x_1x_1}
	\end{array}\right)=\mathbf{0}$.
	Collecting the preceding two results in \eqref{aef} gives
	\begin{eqnarray}
	\mbox{div}(A\mathbf{E})
	&=&-4\mbox{det}(A)\mbox{det}(\nabla^2p)\nonumber\\
	&&+2\mbox{div}(wA\nabla p)+2(\nabla\mbox{det}(A))^T\left(\begin{array}{cc}
	-p_{x_2x_2}&p_{x_1x_2}\\
	p_{x_1x_2}&
	-p_{x_1x_1}
	\end{array}\right)\nabla p.\label{identy2}
	\end{eqnarray} 
	As for $A\mathbf{E}\cdot\mathbf{E}$, we have
	\begin{eqnarray*}
		A\mathbf{E}\cdot\mathbf{E}&=&\mathbf{E}^TA\mathbf{E}\nonumber\\
		&=&4(\nabla p)^TA\nabla^2pA\nabla^2pA\nabla p\nonumber\\
		&=&4(\nabla p)^T(wA-\mbox{det}(A)\mbox{det}(\nabla^2 p).(\nabla^2 p)^{-1})\nabla^2pA\nabla p\nonumber\\
		&=&4(\nabla p)^TwA\nabla^2pA\nabla p-4(\nabla p)^T\mbox{det}(A)\mbox{det}(\nabla^2 p)A\nabla p\nonumber\\
		&=&2wA\mathbf{E}\cdot\nabla p-4v\mbox{det}(A)\mbox{det}(\nabla^2 p).
	\end{eqnarray*}
	We are ready to calculate
	\begin{eqnarray*}
		\lefteqn{v^{-1}\mbox{div}(A\mathbf{E})-v^{-2}A\mathbf{E}\cdot\mathbf{E}}\nonumber\\
		&=&-4v^{-1}\mbox{det}(A)\mbox{det}(\nabla^2p)+2v^{-1}\mbox{div}(wA\nabla p)\nonumber\\
		&&+2v^{-1}(\nabla\mbox{det}(A))^T\left(\begin{array}{cc}
			-p_{x_2x_2}&p_{x_1x_2}\\
			p_{x_1x_2}&
			-p_{x_1x_1}
		\end{array}\right)\nabla p\nonumber\\
		&&-v^{-2}(2wA\mathbf{E}\cdot\nabla p-4v\mbox{det}(A)\mbox{det}(\nabla^2 p))\nonumber\\
		&=&2v^{-1}\mbox{div}(wA\nabla p)-2v^{-2}wA\nabla p\cdot\mathbf{E}\nonumber\\
		&&+2v^{-1}\left(\begin{array}{cc}
			-p_{x_2x_2}&p_{x_1x_2}\\
			p_{x_1x_2}&
			-p_{x_1x_1}
		\end{array}\right)\nabla p\cdot\nabla\mbox{det}(A).
	\end{eqnarray*}
	We still need to eliminate the second partial derivatives of $p$ in the last term of of the preceding equation. If $ \mbox{det}(A)$ had been $1$, then this term would be zero, and hence the proof would conclude. 
	Since we do not have the benefit, we need to continue. In view of \eqref{vddef}, we can deduce from \eqref{edf} and \eqref{eqfp} that
	\begin{eqnarray}
	2\nu_1p_{x_1x_1}+2\nu_2 p_{x_1x_2}&=&e_1,\label{pp1}\\
	2\nu_1p_{x_1x_2}+2\nu_2&=&e_2,\label{pp2}\\
	a_{11}p_{x_1x_1}+2a_{12}p_{x_1x_2}+a_{22}p_{x_2x_2}&=& w.\label{pp3}
	\end{eqnarray}
	Denote by $E$ the coefficient matrix of the above system. Then
	\begin{eqnarray*}
		\mbox{det}E&=&\mbox{det}\left(\begin{array}{ccc}
			2\nu_1&2\nu_2&0\\
			0&2\nu_1&2\nu_2\\
			a_{11}&2a_{12}&a_{22}\end{array}\right)\nonumber\\
		&=&2\nu_1(2\nu_1a_{22}-4\nu_2a_{12})+4a_{11}\nu_2^2\nonumber\\
		&=&2(a_{11}p_{x_1}+a_{12}p_{x_2})\left[2a_{22}(a_{11}p_{x_1}+a_{12}p_{x_2})-4a_{12}(a_{12}p_{x_1}+a_{22}p_{x_2})\right]\nonumber\\
		&&+4a_{11}(a_{12}p_{x_1}+a_{22}p_{x_2})^2\nonumber\\
		&=&4(a_{11}p_{x_1}+a_{12}p_{x_2})\left[(a_{22}a_{11}-2a_{12}^2)p_{x_1}-a_{22}a_{12}p_{x_2}\right]\nonumber\\
		&&+4a_{11}(a_{12}p_{x_1}+a_{22}p_{x_2})^2\nonumber\\
		&=&4a_{11}(a_{22}a_{11}-a_{12}^2)p_{x_1}^2 +8(a_{11}a_{22}a_{12}-a_{12}^3)p_{x_1}p_{x_2}\nonumber\\
		&&+4(a_{11}a_{22}^2-a_{22}a_{12}^2)p_{x_2}^2\nonumber\\
		&=&4\mbox{det}(A)\left(a_{11}p_{x_1}^2 +2a_{12}p_{x_1}p_{x_2}+a_{22}p_{x_2}^2\right)\nonumber\\
		&=&4\mbox{det}(A)v\ne 0.
	\end{eqnarray*}
	By Cramer's rule, we have
	\begin{eqnarray*}
		p_{x_1x_1}
		&=&\frac{1}{2\mbox{det}(A)v}\left[((a_{22}a_{11}-2a_{12}^2)p_{x_1}-a_{12}a_{22}p_{x_2})e_1-a_{22}\nu_2e_2+2w\nu_2^2\right]\nonumber,\\
		p_{x_1x_2}&=&\frac{1}{2\mbox{det}(A)v}\left(a_{11}\nu_2e_1+a_{22}\nu_1e_2-2w\nu_1\nu_2\right),\\
		p_{x_2x_2}&=&\frac{1}{2\mbox{det}(A)v}\left[-a_{11}\nu_1e_1+(-a_{12}a_{11}p_{x_1}+(a_{11}a_{22}-2a_{12}^2)p_{x_2})e_2+2w\nu_1^2\right].
	\end{eqnarray*}
	Recall \eqref{a1}-\eqref{a3}. We can represent the above solution to \eqref{pp1}-\eqref{pp3} in the form
	\begin{eqnarray*}
		\left(\begin{array}{cc}
			-p_{x_2x_2}&p_{x_1x_2}\\
			p_{x_1x_2}&-p_{x_1x_1}
		\end{array}\right)
		&=&\frac{1}{2\mbox{det}(A)v}\left(A_1\mathbf{E}, A_2\mathbf{E}\right)+\frac{w}{\mbox{det}(A)v}A_3.
	\end{eqnarray*}
	To summarize, we have
	\begin{eqnarray*}
		\lefteqn{	\mbox{div}(A\nabla\phi)}\nonumber\\
		&=&v^{-1}\mbox{div}(A\mathbf{E})-v^{-2}A\mathbf{E}\cdot\mathbf{E}-v^{-1}A\mathbf{G}\cdot\mathbf{E}+\mbox{div}(A\mathbf{G})\nonumber\\
		&=&2v^{-1}\mbox{div}(wA\nabla p)-2v^{-2}wA\nabla p\cdot\mathbf{E}\nonumber\\
		&&+2v^{-1}\left(\begin{array}{cc}
			-p_{x_2x_2}&p_{x_1x_2}\\
			p_{x_1x_2}&
			-p_{x_1x_1}
		\end{array}\right)\nabla p\cdot\nabla\mbox{det}(A)\nonumber\\
		&&-v^{-1}A\mathbf{G}\cdot\mathbf{E}+\mbox{div}(A\mathbf{G})\nonumber\\
		&=&2v^{-1}\mbox{div}(wA\nabla p)+\left(-2v^{-2}wA\nabla p-v^{-1}A\mathbf{G}\right)\cdot\mathbf{E}\nonumber\\
		&&+2v^{-1}\left(\frac{1}{2\mbox{det}(A)v}\left(A_1\mathbf{E}, A_2\mathbf{E}\right)+\frac{w}{\mbox{det}(A)v}A_3\right)\nabla p\cdot\nabla\mbox{det}(A)+\mbox{div}(A\mathbf{G})\nonumber\\
		&=&2v^{-1}wA\nabla p\cdot\nabla\phi+\frac{2w}{\mbox{det}(A)v^2}A_3\nabla p\cdot\nabla\mbox{det}(A)+\mbox{div}(A\mathbf{G}+2v^{-1}wA\nabla p)\nonumber\\
		&&+\left(-2v^{-2}wA\nabla p+\frac{1}{\mbox{det}(A)v^2}\left(\begin{array}{c}(\nabla p)^TA_1\\(\nabla p)^TA_2\end{array}\right)\nabla\mbox{det}(A)\right)\cdot v(\nabla\phi-\mathbf{G})\nonumber\\
		&&-v^{-1}A\mathbf{G}\cdot v(\nabla\phi-\mathbf{G})\nonumber\\
		&=&\left(\frac{1}{\mbox{det}(A)v}\left(\begin{array}{c}(\nabla p)^TA_1\\(\nabla p)^TA_2\end{array}\right)\nabla\mbox{det}(A)-A\mathbf{G}\right)\cdot \nabla\phi\nonumber\\
		&&+\left(2v^{-1}wA\nabla p-\frac{1}{\mbox{det}(A)v}\left(\begin{array}{c}(\nabla p)^TA_1\\(\nabla p)^TA_2\end{array}\right)\nabla\mbox{det}(A)+A\mathbf{G}\right)\cdot \mathbf{G}\nonumber\\
		&&+\frac{2w}{\mbox{det}(A)v^2}A_3\nabla p\cdot\nabla\mbox{det}(A)+\mbox{div}(A\mathbf{G}+2v^{-1}wA\nabla p)\nonumber\\
		&=&\mathbf{H}\cdot\nabla\phi+h+\mbox{div}\mathbf{K}.
	\end{eqnarray*}
	This completes the proof. 
\end{proof}

We would like to remark that the last part in our proof only works for two space dimensions. If the space dimension had been three, we would have six second order partial derivatives. But \eqref{edf} and \eqref{eqfp} would only give us four equations. Thus the same argument would fail. However, in the context of our proof, the last part becomes necessary only because we cannot normalize the coefficient matrix. Even if we could have done this, our argument would still only work for the two space dimensions. As we can easily see, it is not possible to represent $A\mathbf{E}\cdot\mathbf{E}$ in terms of $\mbox{det}\left(\nabla^2p\right)$ if the space dimensions are bigger than or equal three.
\begin{theorem}
	For each $r>1$ 
	there is a positive number $c$ such that
	\begin{equation}\label{npb}
	\|v\|_{\infty,\Omega }\leq  c\|\nabla\m\|^{2r}_{\infty, \Omega}+c.
	\end{equation}
\end{theorem}
\begin{proof}
	Recall from \eqref{adf} that
	\begin{equation*}
	A=I+\m\m^T=\left(\begin{array}{cc}
	1+m_1^2&m_1m_2\\
	m_1m_2&1+m_2^2
	\end{array}\right),
	\end{equation*}
	and hence
	%
	%
	\begin{eqnarray}\label{ellip}
	|\mathbf{Y}|^2\leq A\mathbf{Y}\cdot \mathbf{Y}&\leq& (1+|\m|^2)|\mathbf{Y}|^2\ \ \mbox{for each $\mathbf{Y}\in \mathbb{R}^2$}.
	\end{eqnarray}
	It immediately follows that
	\begin{eqnarray}
	|\nabla p|^2\leq v&=&A\nabla p\cdot\nabla p=|\nabla p|^2+(\m\cdot\nabla p)^2\leq (1+|\m|^2)|\nabla p|^2,\label{coa}\\
	\mbox{det}A&=&1+|\m|^2.\nonumber
	\end{eqnarray}
	With these in mind, we can derive that
	\begin{eqnarray*}
		|\mathbf{G}|&\leq &c|\m||\nabla\m|,\\
		|w|&\leq & c|\m||\nabla\m||\nabla p|+|s(x)|,\\
		|A_1|,|A_2| &\leq &c(1+|\m|^4)|\nabla p|,\\
		|A_3| &\leq &c(1+|\m|^4)|\nabla p|^2.
	\end{eqnarray*}
	Let
	\begin{equation}\label{ddef}
	d=(1+|\m|^2)|\m||\nabla\m|.
	\end{equation}
	We can easily deduce that
	\begin{eqnarray}
	|\mathbf{H}|&\leq &cd,\label{ccoe1}\\
	|\mathbf{K}|&\leq &cd+c(1+|\m|^2)\frac{|s(x)|}{|\nabla p|},\label{ccoe2}\\
	|h|&\leq &cd|\m||\nabla\m|+\frac{cd|s(x)|}{|\nabla p|}.\label{ccoe3}
	\end{eqnarray}
	In addition, \eqref{coa} implies
	\begin{equation}\label{ccoe4}
	|\nabla p|^2\geq \frac{1}{1+|\m|^2}\ \ \mbox{on $\{v\geq 1\}$}.
	\end{equation}
	Hence,
	\begin{eqnarray}
	|\mathbf{K}|&\leq& cd+c(1+|\m|^2)^{\frac{3}{2}}|s(x)| \ \ \mbox{on $\{v\geq 1\}$},\label{ccoe5}\\
	|h|&\leq& cd|\m||\nabla\m|+cd\sqrt{1+|\m|^2}|s(x)|\nonumber\\
	&\leq & cd^2+c(1+|\m|^2)|s^2(x)\ \ \mbox{on $\{v\geq 1\}$} .\label{ccoe6}
	\end{eqnarray}
	
	Now fix a point $x_0\in \Omega$. Then pick a number $R$ from $(0,\mbox{dist}(x_0,\partial\Omega))$. Define a sequence of concentric balls $B_{R_n}(x_0)$ in $\Omega$ as follows:
	\begin{equation*}
	B_{R_n}(x_0)=\{x:|x-x_0|<R_n\},
	\end{equation*}
	where
	\begin{equation*}
	R_n= \frac{R}{2}+\frac{R}{2^{n+1}},\ n=0,1,2,\cdots.
	\end{equation*} Choose a sequence of smooth functions $\theta_n$ so that
	\begin{eqnarray*}
		\theta_n(x)&=& 1 \ \ \mbox{in $B_{R_n}(x_0)$},\\
		\theta_n(x)&=&0\ \ \mbox{outside $B_{R_{n-1}}(x_0)$},\\
		|\nabla \theta_n(x)|&\leq & \frac{c2^n}{R}\ \ \mbox{for each $x\in \mathbb{R}^2$,}\ \ \ \mbox{and}\\
		0&\leq &\theta_n(x)\leq 1\ \ \mbox{in $\mathbb{R}^2$.}
	\end{eqnarray*}
	Select
	\begin{equation}\label{kbt}
	K\geq2
	\end{equation}
	as below.
	Set
	\begin{equation*}
	K_n=K-\frac{K}{2^{n+1}},\ \ \ n=0,1,2,\cdots.
	\end{equation*}
	Hence,
	\begin{equation*}
	K_n\geq 1\ \ \mbox{ for each $n$}.
	\end{equation*} 
	We use $\theta_{n+1}^2( v -K_{n+1})^+$  as a test function in \eqref{wine11} to obtain
	\begin{eqnarray}
	\lefteqn{\io\frac{1}{ v } A\nabla  v \cdot\nabla ( v -K_{n+1})^+\theta_{n+1}^2dx}\nonumber\\
	&=&-2\io\frac{1}{ v }  A\nabla  v \cdot\nabla\theta_{n+1}( v -K_{n+1})^+\theta_{n+1}dx\nonumber\\
	&&-\io\frac{1}{ v }\mathbf{H}\nabla v \theta_{n+1}^2( v -K_{n+1})^+dx
	-\io h\theta_{n+1}^2( v -K_{n+1})^+dx\nonumber\\
	&&+\io \mathbf{K}\cdot\nabla ( v -K_{n+1})^+\theta_{n+1}^2dx\nonumber\\
	&&+2\io\mathbf{K} \cdot\nabla\theta_{n+1}( v -K_{n+1})^+\theta_{n+1}dx.\label{wine6}
	\end{eqnarray}
	Note that
	\begin{equation*}
	\nabla v =\nabla ( v -K_{n+1})^+\ \ \mbox{on $S_{n+1}(t)$},
	\end{equation*}
	where
	\begin{equation*}
	S_{n+1}(t)=\{x\in B_n(x_0):  v (x, t)\geq K_{n+1}\}.
	\end{equation*}
	This together with \eqref{ellip} and \eqref{wine6} implies
	\begin{eqnarray}
	\lefteqn{\io\frac{1}{ v } |\nabla ( v -K_{n+1})^+|^2\theta_{n+1}^2dx}\nonumber\\
	&\leq&\frac{c4^n}{R^2}\int_{S_{n+1}(t)}\frac{1}{ v } | A|\left[( v -K_{n+1})^+\right]^2dx\nonumber\\
	&&+\io\frac{c}{ v }|\mathbf{H}|^2\theta_{n+1}^2\left[( v -K_{n+1})^+\right]^2dx
	+ \io| h|\theta_{n+1}^2( v -K_{n+1})^+dx\nonumber\\
	&&+\int_{S_{n+1}(t)} c v |\mathbf{K}|^2\theta_{n+1}^2dx+\frac{c2^n}{R}\io|\mathbf{K}| ( v -K_{n+1})^+\theta_{n+1}dx.\label{wine12}
	\end{eqnarray}
	We proceed to analyze each term on the right-hand side of the above inequality.
	Note that
	\begin{equation*}
	|A|\leq 1+|\m|^2.
	\end{equation*}
	The last term in \eqref{wine12} can be estimated as follows:
	\begin{eqnarray*}
		\frac{2^n}{R}\io|\mathbf{K}| ( v -K_{n+1})^+\theta_{n+1}dx&\leq& \frac{c4^n}{R^2}\int_{S_{n+1}(t)}\frac{1}{ v } \left[( v -K_{n+1})^+\right]^2dx\nonumber\\
		&&+\int_{S_{n+1}(t)} c v |\mathbf{K}|^2\theta_{n+1}^2dx.
	\end{eqnarray*}
	Thus this integral can be dropped. 
	Observe that
	\begin{eqnarray*}
		\frac{1}{ v } |\nabla ( v -K_{n+1})^+|^2&=&4|\nabla(\sqrt{ v }-\sqrt{K_{n+1}})^+|^2,\\
		\frac{1}{ v } \left[( v -K_{n+1})^+\right]^2&=&\frac{1}{ v }\left[\left(\sqrt{ v }-\sqrt{K_{n+1}}\right)^+\right]^2\left(\sqrt{ v }+\sqrt{K_{n+1}}\right)^2\nonumber\\
		&=&\left[\left(\sqrt{ v }-\sqrt{K_{n+1}}\right)^+\right]^2\left(1+\frac{\sqrt{K_{n+1}}}{\sqrt{ v }}\right)^2\nonumber\\
		&\leq &4\left[\left(\sqrt{ v }-\sqrt{K_{n+1}}\right)^+\right]^2.
	\end{eqnarray*}
	Notice that
	\begin{eqnarray*}
		\frac{\sqrt{K_{n+1}}-\sqrt{K_{n}}}{\sqrt{K_{n+1}}}&=&\frac{\sqrt{1-\frac{1}{2^{n+2}}}-\sqrt{1-\frac{1}{2^{n+1}}}}{\sqrt{1-\frac{1}{2^{n+2}}}}\nonumber\\
		&=&\frac{1}{2^{n+2}\left(\sqrt{1-\frac{1}{2^{n+2}}}+\sqrt{1-\frac{1}{2^{n+1}}}\right)\sqrt{1-\frac{1}{2^{n+2}}}}\nonumber\\
		&\geq &\frac{1}{2^{n+3}}.
	\end{eqnarray*}
	With this in mind, we estimate
	\begin{eqnarray*}
		\left[\left(\sqrt{ v }-\sqrt{K_{n}}\right)^+\right]^2&\geq &\left[\left(\sqrt{ v }-\sqrt{K_{n}}\right)^+\right]^2\chi_{S_{n+1}(t)}\nonumber\\
		&=&\frac{1}{2}\left(\sqrt{ v }-\sqrt{K_{n}}\right)^+(\sqrt{ v }+\sqrt{ v })\left(1-\frac{\sqrt{K_{n}}}{\sqrt{ v }}\right)\chi_{S_{n+1}(t)}\nonumber\\
		&\geq &\frac{1}{2}\left(\sqrt{ v }-\sqrt{K_{n}}\right)^+(\sqrt{ v }+\sqrt{K_{n+1}})\left(1-\frac{\sqrt{K_{n}}}{\sqrt{K_{n+1}}}\right)\chi_{S_{n+1}(t)}\nonumber\\
		&\geq &\frac{1}{2^{n+4}}( v -K_{n+1})^+.
	\end{eqnarray*}
	Here $\chi_{S_{n+1}(t)}$ is the indicator function of the set $S_{n+1}(t)$.
	Similarly,
	\begin{equation*}
	\left[\left(\sqrt{ v }-\sqrt{K_{n}}\right)^+\right]^2\geq  v \left[\left(1-\frac{\sqrt{K_{n}}}{\sqrt{ v }}\right)^+\right]^2\chi_{S_{n+1}(t)}\geq \frac{1}{2^{2(n+3)}} v \chi_{S_{n+1}(t)}.
	\end{equation*}
	Plugging the preceding results into \eqref{wine12}, we obtain
	\begin{eqnarray}
	\lefteqn{\io|\nabla(\sqrt{ v }-\sqrt{K_{n+1}})^+|^2\theta_{n+1}^2dx}\nonumber\\
	&\leq&\frac{c4^n}{R^2}\int_{S_{n+1}(t)} (1+|\m|^2)\left[\left(\sqrt{ v }-\sqrt{K_{n+1}}\right)^+\right]^2dx\nonumber\\
	&&+c\io|\mathbf{H}|^2\left[\left(\sqrt{ v }-\sqrt{K_{n+1}}\right)^+\right]^2\theta_{n+1}^2dx
	+c2^n\io| h|\theta_{n+1}^2\left[\left(\sqrt{ v }-\sqrt{K_{n}}\right)^+\right]^2dx\nonumber\\
	&&+c2^{2n}\int_{S_{n+1}(t)} |\mathbf{K}|^2\left[\left(\sqrt{ v }-\sqrt{K_{n}}\right)^+\right]^2\theta_{n+1}^2dx.\label{wine14}
	\end{eqnarray}
	We pick a number $r$ from the interval $ (1, \infty)$. 
	Define
	\begin{equation*}
	y_n=\left(\int_{B_{R_n}(x_0)}\left[\left(\sqrt{ v }-\sqrt{K_{n}}\right)^+\right]^{2r}dx\right)^{\frac{1}{r}}.
	\end{equation*}
	We conclude from \eqref{wine14} that
	\begin{eqnarray}
	\lefteqn{\io|\nabla(\sqrt{ v }-\sqrt{K_{n+1}})^+|^2\theta_{n+1}^2dx}\nonumber\\
	&\leq&\frac{c4^ny_n}{R^2} \| (1+|\m|^2)\|_{\frac{r}{r-1},S_{1}(t)}+ cy_n\| \mathbf{H}\|^2_{\frac{2r}{r-1},S_{1}(t)}
	+c2^ny_n\|h\|_{\frac{r}{r-1},S_{1}(t)}\nonumber\\
	&&+c2^{2n}y_n\| \mathbf{K}\|^2_{\frac{2r}{r-1},S_{1}(t)}\nonumber\\
	&\leq&\frac{c4^ny_n}{R^2}\Gamma ,\label{wine15}
	\end{eqnarray}
	where
	\begin{equation*}
	\Gamma=\| (1+|\m|^2)\|_{\frac{r}{r-1},S_{1}(t)}+R^2\left(\| \mathbf{H}\|^2_{\frac{2r}{r-1},S_{1}(t)}+\|h\|_{\frac{r}{r-1},S_{1}(t)}+\| \mathbf{K}\|^2_{\frac{2r}{r-1},S_{1}(t)}\right).
	\end{equation*}
	By Poincar\'{e}'s inequality, we have
	\begin{eqnarray}
	y_{n+1}&\leq &\left(\io\left((\sqrt{ v }-\sqrt{K_{n+1}})^+\theta_{n+1}\right)^{2r}dx \right)^{\frac{1}{r}}\nonumber\\
	&\leq &c\left(\io\left|\nabla\left((\sqrt{ v }-\sqrt{K_{n+1}})^+\theta_{n+1}\right)\right|^{\frac{2r}{r+1}}dx \right)^{\frac{r+1}{r}}\nonumber\\
	&\leq &c\io\left|\nabla\left((\sqrt{ v }-\sqrt{K_{n+1}})^+\theta_{n+1}\right)\right|^{2}dx |S_{n+1}(t)|^{\frac{1}{r}}\nonumber\\
	&\leq &c\io\left|\nabla\left((\sqrt{ v }-\sqrt{K_{n+1}})^+\right|^{2}\theta_{n+1}^2\right)dx |S_{n+1}(t)|^{\frac{1}{r}}\nonumber\\
	&&+\frac{c4^{n}}{R^2}\int_{B_{R_n}(x_0)}\left[\left(\sqrt{ v }-\sqrt{K_{n+1}}\right)^+\right]^{2}dx|S_{n+1}(t)|^{\frac{1}{r}}\nonumber\\
	&\leq &\frac{c4^n}{R^2}\Gamma y_n |S_{n+1}(t)|^{\frac{1}{r}}+\frac{c4^{n}}{R^2} y_nR^{\frac{2(r-1)}{r}}|S_{n+1}(t)|^{\frac{1}{r}}\nonumber\\
	&=&\frac{c4^n}{R^2}\left(\Gamma+R^{\frac{2(r-1)}{r}}\right)y_n |S_{n+1}(t)|^{\frac{1}{r}}.\label{wine16}
	\end{eqnarray}
	We easily see that
	\begin{equation*}
	y_n\geq \left(\int_{S_{n+1}(t)}(\sqrt{K_{n+1}}-\sqrt{K_{n}})^{2r}dx\right)^{\frac{1}{r}}\geq\frac{K}{2^{2(n+3)}}|S_{n+1}(t)|^{\frac{1}{r}}.
	\end{equation*}
	Substituting this into \eqref{wine16} yields
	\begin{equation*}
	y_{n+1}\leq\frac{c4^n}{R^2K} \left(\Gamma+R^{\frac{2(r-1)}{r}}\right)y_n^2.
	\end{equation*}
	In view of Lemma \ref{ynb}, if we choose $K$ so large that
	\begin{equation*}
	y_0\leq \frac{cR^2K}{\Gamma+R^{\frac{2(r-1)}{r}}},
	\end{equation*} then we have
	\begin{equation*}
	\lim_{n\rightarrow 0}y_n=0.
	\end{equation*}
	Taking into account \eqref{kbt}, it is enough for us to let
	\begin{equation*}
	K= \frac{c}{R^2}y_0\left(\Gamma+R^{\frac{2(r-1)}{r}}\right)+2.
	\end{equation*}
	We arrive at
	\begin{equation}\label{wine19}
	\sup_{B_{ \frac{R}{2}}(x_0)} v \leq K= \frac{c}{R^2}y_0\left(\Gamma+R^{\frac{2(r-1)}{r}}\right)+2.
	\end{equation}
	Now we proceed to estimate $\Gamma$.
	Combing \eqref{mint} with \eqref{ddef} and \eqref{ccoe1} yields that
	\begin{equation*}
	\|\mathbf{H}\|^2_{\frac{2r}{r-1}, S_1(t)}\leq c\|\nabla\m\|^2_{\infty, B_{R}(x_0)}.
	\end{equation*}
	Similarly, by (H1), \eqref{ccoe5}, and \eqref{ccoe6}, we have
	\begin{eqnarray*}
		\|\mathbf{K}\|^2_{\frac{2r}{r-1}, S_1(t)}&\leq c&\|\nabla\m\|^2_{\infty, B_{R}(x_0)}+c\|s^2(x)\|_{\infty, B_{R}(x_0)},\nonumber\\
		\|h\|_{\frac{r}{r-1}, S_1(t)}&\leq c&\|\nabla\m\|^2_{\infty, B_{R}(x_0)}+c\|s^2(x)\|_{\infty, B_{R}(x_0)}.
	\end{eqnarray*}
	Furthermore,
	\begin{equation*}
	y_0=\left(\int_{B_{ R}(x_0)}\left[\left(\sqrt{ v }-\sqrt{\frac{K}{2}}\right)^+\right]^{2r}dx\right)^{\frac{1}{r}}
	\leq \|v\|_{ r, B_{ R}(x_0)}.
	\end{equation*}
	Collecting the preceding estimates in \eqref{wine19}, we arrive at
	\begin{equation}\label{wine20}
	\sup_{B_{\frac{R}{2}}(x_0)} v\leq  c\|v\|_{ r, B_{ R}(x_0)}\left(\|\nabla\m\|^2_{\infty, B_{R}(x_0)}+1+\frac{1}{R^2}\right)+c.
	\end{equation}
	By an argument in (\cite{GT}, p. 303), we can extend the above estimate to the whole $\Omega$. That is, we have
	\begin{equation}\label{wine21}
	\sup_{\Omega} v\leq  c\|v\|_{r, \Omega}\left(\|\nabla\m\|^2_{\infty, \Omega}+1\right)+c.
	\end{equation}
	A more detailed explanation for this can be found in \cite{X7}.
	Remember from \eqref{f5} that
	\begin{equation*}
	\int_{\Omega}vdx=\int_{\Omega}\left(|\nabla p|^2+(\mnp)^2\right)dx\leq c.
	\end{equation*}
	On account of \eqref{inter}, we have
	\begin{eqnarray*}
		\|v\|_{r,\Omega }
		&\leq &\varepsilon\|v\|_{\infty,\Omega }+\frac{1}{\varepsilon^{r-1}}\|v\|_{1,\Omega }\nonumber\\
		&\leq &\varepsilon\|v\|_{\infty,\Omega }+\frac{c}{\varepsilon^{r-1}},\ \ \varepsilon>0.
	\end{eqnarray*}
	Substituting this into \eqref{wine21} and choosing $\varepsilon$ suitably small in the reulting inequality, we obtain \eqref{npb}.
\end{proof}

We are ready to prove the main theorem.
\begin{proof}[Proof of the Main Theorem]
	Define
	\begin{equation*}
	f_i(x,t)=\left\{\begin{array}{ll}
	\beta^2\m\cdot\nabla p p_{x_i}-|\m|^{2(\gamma-1)}m_i&\mbox{if $(x,t)\in\ot$},\\
	0&\mbox{if $(x,t)$ lies outside $\ot$}.
	\end{array}\right.
	\end{equation*}
	Consider the function
	\begin{equation*}
	u_i=\frac{1}{ 4\pi\alpha^2}\int_{0}^{t}\frac{1}{t-\tau}\int_{\mathbb{R}^2}\exp\left(-\frac{|x-y|^2}{4\alpha^2(t-\tau)}\right)f_i(y,\tau)dyd\tau.
	\end{equation*}
	We see from (\cite{LSU}, Chapter IV) that $u_i$ satisfies
	\begin{eqnarray*}
		(u_i)_t-\alpha^2\Delta u_i&=& f_i\ \ \ \mbox{in $\mathbb{R}^2\times(0,\infty)$,}\\
		u_i(x,0)&=&0\ \ \mbox{on $\mathbb{R}^2$}.
	\end{eqnarray*}
	Furthermore, for each $s> 1$ there is a positive number $c$ such that
	\begin{equation}\label{lsu1}
	\|(u_i)_t\|_{s,\ot}+ \|u_i\|_{L^s(0,T;W^{2,s}(\Omega))}\leq c\|f_i\|_{s,\ot}.
	\end{equation}
	We infer from  \eqref{mint} that for each $\delta_0\in (1,2)$ and $s\geq 1$ there is a positive number such that
	\begin{equation*}
	\sup_{0\leq t\leq T}\io\frac{|\m|^s}{|x-y|^{\delta_0}}dy\leq c.
	\end{equation*}
	Set
	\begin{equation*}
	l=\frac{|x-y|}{2\alpha\sqrt{t-\tau}}.
	\end{equation*}
	For each $\delta\in (2,3)$ we estimate
	\begin{eqnarray}
	|\nabla u_i|&=&\left|\frac{1}{ 16\pi\alpha^4}\int_{0}^{t}\frac{1}{(t-\tau)^2}\int_{\mathbb{R}^2}(x-y)\exp\left(-l^2\right)f_i(y,\tau)\frac{}{}dyd\tau\right|\nonumber\\
	&\leq &c\int_{0}^{t}\frac{1}{(t-\tau)^2}\int_{\mathbb{R}^2}\frac{\left(2\alpha\sqrt{t-\tau}\right)^\delta}{|x-y|^{\delta-1}}l^\delta\exp\left(-l^2\right)|f_i(y,\tau)|dyd\tau\nonumber\\
	&\leq &c\|\nabla p\|^2_{\infty, \ot}\int_{0}^{t}\frac{1}{(t-\tau)^{2-\frac{\delta}{2}}}\int_{\mathbb{R}^2}|\m|\chi_{\ot}\frac{1}{|x-y|^{\delta-1}}dyd\tau\nonumber\\
	&&+\int_{0}^{t}\frac{1}{(t-\tau)^{2-\frac{\delta}{2}}}\int_{\mathbb{R}^2}|\m|^{2\gamma-1}\chi_{\ot}\frac{1}{|x-y|^{\delta-1}}dyd\tau\nonumber\\
	&\leq &\left(c\|\nabla p\|^2_{\infty, \ot}+c\right)t^{-1+\frac{\delta}{2}}.\label{nmb2}
	\end{eqnarray}
	Obviously, $w_i\equiv m_i-u_i$ satisfies the problem
	\begin{eqnarray}
	(w_i)_t-\alpha^2\Delta w_i&=&0\ \ \mbox{in $\ot$},\\
	w_i&=&-u_i\ \ \mbox{on $\Sigma_T$,}\\
	w_i&=&m_{0i}.
	\end{eqnarray}
	We can easily conclude from \eqref{lsu1} and the classical regularity theory for the heat equation (\cite{LSU}, Chapter IV) that
	$\|\nabla w_i\|_{\infty, \ot}\leq c\|\nabla u_i\|_{\infty, \ot}+c\|\nabla m_{0i}\|_{\infty, \ot}$.
	This together with \eqref{nmb2} gives
	\begin{equation}\label{nmb1}
	\|\nabla \m\|_{\infty, \ot}\leq cT^{-1+\frac{\delta}{2}}\left(\|\nabla p\|^2_{\infty, \ot}+1\right)+c.
	\end{equation}
	We deduce from \eqref{npb} that
	\begin{eqnarray}\label{nmb}
	\|\nabla p\|^2_{\infty,\ot }&\leq& \|v\|_{\infty,\ot }\leq c\|\nabla\m\|^{2r}_{\infty,\ot}+c\nonumber\\
	&\leq &cT^{r(\delta-2)}\left(\|\nabla p\|^{4r}_{\infty, \ot}+1\right)+c.
	\end{eqnarray}
	Set \begin{equation*}
	\ep=cT^{r(\delta-2)}.
	\end{equation*}
	Consider the function $g(\tau)=\ep \tau^{2r}-\tau+\ep+c$ on $[0, \infty)$. Then \eqref{nmb} asserts 
	\begin{equation}\label{key}
	g\left(\|\nabla p\|^2_{\infty,\Omega\times[0,s] }\right)\geq 0\ \ \mbox{for each $s\in [0,T]$}.
	\end{equation}The function $g$ achieves its minimum value at $\tau_0=\frac{1}{\left(2\ep r\right)^{\frac{1}{2r-1}}}$. The minimum value
	\begin{eqnarray*}
		g(\tau_0)&=&\frac{\ep}{\left(2\ep r\right)^{\frac{2r}{2r-1}}}-\frac{1}{\left(2\ep r\right)^{\frac{1}{2r-1}}}+\ep+c\nonumber\\
		&=&\ep+c-\frac{\ep(2r-1)}{\left(2\ep r\right)^{\frac{2r}{2r-1}}}\leq -\ep,
	\end{eqnarray*}
	provided that
	\begin{equation}\label{cont1}
	(c+2\ep)\ep^{\frac{1}{2r-1}} \leq\frac{2r-1}{(2r)^{\frac{2r}{2r-1}}}.
	\end{equation}  In addition to this, we require that
	\begin{equation}\label{cont2}
	\|\nabla p(\cdot,0)\|^2_{\infty,\Omega}\leq \tau_0.
	\end{equation}  
	Recall that our solution $(\m, p)$ here is the local solution constructed in \cite{X6}. It is not difficult to see that $\|\nabla p(\cdot,t)\|_{\infty,\Omega}$ is a continuous function of $t$. We can conclude from \eqref{key} that $\|\nabla p\|^2_{\infty,\ot}\leq \tau_0$ whenever both \eqref{cont1} and \eqref{cont2} hold. Condition \eqref{cont1} can be achieved easily by taking $T$ suitably small. As for \eqref{cont2}, recall that $p(x,0)=p_0$ satisfies the boundary value problem \eqref{po1}-\eqref{po2}. Our assumptions on $\m_0, s(x)$ are strong enough to guarantee that $|\nabla p_0|$ is bounded. We can also obtain \eqref{cont2} for suitably small $T$. In summary, if  $T_0$ is the largest $T$ such that both \eqref{cont1} and \eqref{cont2} hold, then
	\begin{equation}\label{rr1}
	\|\nabla p\|^2_{\infty,\Omega\times[0,T_0]}\leq \tau_0.
	\end{equation}
	We consider $(\m(x,t+T_0), p(x,t+T_0 ))$ on $\Omega\times[0,T_0]$. Conditions \eqref{cont1} and \eqref{cont2} still hold, and so does \eqref{rr1}. Therefore, 
	we can extend the solution in the time direction as far away as we want.
\end{proof}


\end{document}